\documentclass[a4paper,11pt] {amsart}
\usepackage{amssymb,latexsym,amsmath,mathrsfs}
\usepackage{amsfonts,amsbsy,bm}
\usepackage{color}
\usepackage[numbers,sort&compress]{natbib}

\newtheorem{theorem}{Theorem}[section]
\newtheorem{lemma}[theorem]{Lemma}
\newtheorem{corollary}[theorem]{Corollary}
\newtheorem{proposition}[theorem]{Proposition}

\newtheorem{definition}[theorem]{Definition}

\newtheorem{remark}{Remark}

\newcommand{\Ba}[1]{\begin{array}{#1}}
\newcommand{\Ea}{\end{array}}
\newcommand{\Be}{\begin{equation}}
\newcommand{\Ee}{\end{equation}}
\newcommand{\Bea}{\begin{eqnarray}}
\newcommand{\Eea}{\end{eqnarray}}
\newcommand{\Beas}{\begin{eqnarray*}}
\newcommand{\Eeas}{\end{eqnarray*}}

\begin{document}


\title{Carleson embeddings with loss for Bergman-Orlicz spaces of the unit ball}

\address{Beno\^it F. Sehba\\Department of Mathematics, University of Ghana,\\ P. O. Box LG 62 Legon, Accra,Ghana}
\email{bfsehba@ug.edu.gh} 

\maketitle

\begin{abstract}
We prove Carleson embeddings for Bergman-Orlicz spaces of the unit ball that extend the lower triangle estimates for the usual Bergman spaces.
\end{abstract}


\section{Introduction}
\setcounter{equation}{0} \setcounter{footnote}{0}
\setcounter{figure}{0} Our setting is the unit ball $\mathbb B^n$ of $\mathbb{C}^n$. We denote by $H(\mathbb{B}^n)$, the space of all holomorphic functions on $\mathbb{B}^n$. Let us denote by $d\nu$ the Lebesgue measure on $\mathbb B^n$.
 

A surjective function $\Phi:[0,\infty)\rightarrow [0,\infty)$ is a growth function, if it is continuous and non-decreasing. We note that this implies that $\Phi(0)=0$.

For $\alpha>-1$, we denote by $\nu_{\alpha}$ the normalized Lebesgue measure on $\mathbb{B}^n$ defined by $d\nu_{\alpha}(z)=c_{\alpha}(1-|z|^2)^{\alpha}d\nu(z)$, where $c_\alpha$ is the normalizing constant. Given a growth function $\Phi$, we denote by $L_\alpha^{\Phi}(\mathbb B^n)$ the space of all functions $f$ such that
$$||f||_{\Phi,\alpha}=||f||_{L_\alpha^{\Phi}}:=\int_{\mathbb B^n}\Phi(|f(z)|)d\nu_{\alpha}(z)<\infty.$$
The weighted Bergman-Orlicz space $A_\alpha^{\Phi}(\mathbb B^n)$ is the subspace of $L_\alpha^{\Phi}(\mathbb B^n)$ consisting of all holomorphic functions.
We define on $A_\alpha^{\Phi}(\mathbb B^n)$ the following Luxembourg (quasi)-norm
\begin{equation}\label{BergOrdef1}
||f||^{lux}_{\Phi,\alpha}=||f||^{lux}_{L_\alpha^{\Phi}}:=\inf\left\{\lambda>0: \int_{\mathbb B^n}\Phi\left(\frac{|f(z)|}{\lambda}\right)d\nu_{\alpha}(z)\le 1\right\}.
\end{equation}
We observe that $||f||^{lux}_{\Phi,\alpha}$ is finite if $f\in A_\alpha^{\Phi}(\mathbb B^n)$ (see \cite[Remark 1.4]{sehbastevic}). Moreover, $||f||^{lux}_{\Phi,\alpha}=0$ implies that $f=0$ (see \cite[Page 569]{sehbastevic}).


The usual weighted Bergman space $A_\alpha^{p}(\mathbb B^n)$ corresponds to $\Phi(t)=t^p$ and is defined as the set of all $f\in H(\mathbb{B}^n)$ such that
$$||f||_{p,\alpha}^p:= \int_{\mathbb B^n}|f(z)|^pd\nu_{\alpha}(z)<\infty.$$


 Given $0< p,q<\infty$, we consider the question of the characterization of the positive measures $\mu$ on $\mathbb{B}^n$ such that the embedding $I_\mu:A^p_\alpha(\mathbb{B}^n)\rightarrow L^q(\mathbb{B}^n, d\mu)$ is continuous. That is, there exists a constant $C>0$ such that the following inequality
$$\int_{\mathbb{B}^n}|f(z)|^qd\mu(z)\le C\|f\|_{p,\alpha}^q$$
holds for any $f\in A^p_\alpha(\mathbb{B}^n)$, $\alpha>-1$. In the setting of Bergman spaces of the unit disk, for $0<p\le q<\infty$, this question was answered by W. Hastings \cite{hastings} and the answer to the case $0<q<p<\infty$ (estimation with loss) was obtained by D. Luecking \cite{luecking1}. The extensions of these results to the unit ball are due to J. A. Cima and W. Wogen \cite{CW}  and D. Luecking \cite{luecking2,luecking3}. 
When this holds, we speak of Carleson embedding (or estimate) for $A^p_\alpha(\Omega)$; we also say that $\mu$ is a $q$-Carleson measure for $A^p_\alpha(\Omega)$.
\vskip .2cm
Let $\Phi$ be a growth function. For a function $f\in H(\mathbb{B}^n)$, we put 
$$\mathcal{A}_{\Phi,\mu}(f):=\inf\left\{\lambda>0:\,\int_{\mathbb{B}^n}\Phi\left(\frac{|f(z)|}{\lambda}\right)d\mu(z)\le 1\right\}.$$
In the setting of Bergman-Orlicz spaces, the definition of Carleson measures is as follows.
\begin{definition}
Let $\Phi_1$ and $\Phi_2$ be two growth functions, and let $\alpha>-1$. Let $\mu$ be a positive measure on $\mathbb{B}^n$.  
We say $\mu$ is a $\Phi_2$-Carleson measure for $A_\alpha^{\Phi_1}(\mathbb{B}^n)$, if there exists a constant $C>0$ such that for any $f\in A_\alpha^{\Phi_1}(\mathbb{B}^n)$,
\Be\label{eq:carlmeasdef}
\mathcal{A}_{\Phi_2,\mu}(f)\le C\|f\|_{\Phi_1,\alpha}.
\Ee
\end{definition}
\begin{remark}
We observe that (\ref{eq:carlmeasdef}) is equivalent to saying that there is a constant $C>0$ such that for any $f\in A_\alpha^{\Phi_1}(\mathbb{B}^n)$, $f\neq 0$,
\Be\label{eq:carlmeasdefequiv}
\int_{\mathbb{B}^n}\Phi_2\left(\frac{|f(z)|}{\|f\|_{\Phi_1,\alpha}}\right)d\mu(z)\le C.
\Ee
\end{remark}
Our concern in this paper is the characterization of the positive measures $\mu$ such that (\ref{eq:carlmeasdefequiv}) holds, when $\Phi_1$ and $\Phi_2$ are in some appropriate sub-classes of growth functions and such that $\Phi_2\circ\Phi_1^{-1}$ is non-increasing. This corresponds exactly to the case $0<q<p<\infty$ for the usual Carleson measures described above. The case where $\Phi_2\circ\Phi_1^{-1}$ is non-decreasing has been answered in \cite{Charpentier,Charpentiersehba,sehba}.
\vskip .2cm
Carleson measures are an important tool usually used in several problems in mathematical analysis and its applications. Among these problems, one has the questions of the boundedness of Composition operators and Toeplitz operators to name a few.

\section{Settings and presentation of the main result}
We recall that the growth function $\Phi$ is of upper type $q$ if we can find $q > 0$ and $C>0$ such that, for $s>0$ and $t\ge 1$,
\begin{equation}\label{uppertype}
 \Phi(st)\le Ct^q\Phi(s).\end{equation}
We denote by $\mathscr{U}^q$ the set of growth functions $\Phi$ of upper type $q$, (with $q\ge 1$), such that the function $t\mapsto \frac{\Phi(t)}{t}$ is non-decreasing. We write $$\mathscr{U}=\bigcup_{q\geq 1}\mathscr{U}^q.$$

We also recall that $\Phi$ is of lower type $p$ if we can find $p > 0$ and $C>0$ such that, for $s>0$ and $0<t\le 1$,
\begin{equation}\label{lowertype}
 \Phi(st)\le Ct^p\Phi(s).\end{equation}
We denote by $\mathscr{L}_p$ the set of growth functions $\Phi$ of lower type $p$,  (with $p\le 1$), such that the function $t\mapsto \frac{\Phi(t)}{t}$ is non-increasing. We write $$\mathscr{L}=\bigcup_{0<p\leq 1}\mathscr{L}_p.$$
Remark that we may always suppose that any $\Phi\in \mathscr{L}$ (resp. $\mathscr{U}$),  is concave (resp. convex) and
that $\Phi$ is a $\mathscr{C}^1$ function with derivative $\Phi'(t)\approx \frac{\Phi(t)}{t}$.
\vskip .2cm
The complementary function $\Psi$ of the growth function $\Phi$, is the function defined from $\mathbb R_+$ onto itself by
\begin{equation}\label{complementarydefinition}
\Psi(s)=\sup_{t\in\mathbb R_+}\{ts - \Phi(t)\}.
\end{equation}

The growth function $\Phi$ satisfies the $\Delta_2$-condition if there exists a constant $K>1$ such that, for any $t\ge 0$,
\begin{equation}\label{eq:delta2condition}
 \Phi(2t)\le K\Phi(t).\end{equation}

A growth function $\Phi$ is said to satisfy the $\nabla_2$-condition whenever both $\Phi$ and its complementary function satisfy the $\Delta_2$-conditon.

\vskip .2cm
For $z=(z_1,\ldots,z_n)$ and $w=(w_1,\ldots,w_n)$ in $\mathbb{C}^n$, we
let $${\langle z,w\rangle=z_1\overline {w_1} + \cdots + z_n\overline {w_n}}$$ which
gives $|z|^2=\langle z,z\rangle=|z_1|^2 +\cdots +|z_n|^2$.
 \vskip .1cm
 For $z\in \mathbb B^n$ and $\delta>0$ define the average function $$\hat {\mu}_\delta(z):=\frac{\mu(D(z,\delta))}{v_\alpha(D(z,\delta))}$$
 where $D(z,\delta)$ is the Bergman metric ball centered at $z$ with radius $\delta$.
 \vskip .2cm
 The Berezin transform $\tilde {\mu}$  of the measure $\mu$ is the function defined for any $w\in\mathbb{B}^n$ by $$\tilde {\mu}(w):=\int_{\mathbb{B}^n}\frac{(1-|w|^2)^{n+1+\alpha}}{|1-\langle z,w\rangle|^{2(n+1+\alpha)}}d\mu(z).$$

Our main result can be stated as follows.
\begin{theorem}\label{theo:main} Let $\Phi_1,\Phi_2\in \mathscr{L}\cup\mathscr{U}$, $\alpha>-1$.  Assume that 
\begin{itemize}
\item[(i)] $\Phi_1\circ\Phi_2^{-1}$ satisfies the $\nabla_2$-condition;
\item[(ii)] $\frac{\Phi_1\circ\Phi_2^{-1}(t)}{t}$ is non-decreasing.
\end{itemize}
Let $\mu$ be a positive measure on $\mathbb B^n$, and let  $\Phi_3$ be the complementary function of $\Phi_1\circ\Phi_2^{-1}$. Then the following assertions are equivalent.
 \begin{itemize}
 \item[(a)] There is a constant $C>0$ such that for any $f\in A_\alpha^{\Phi_1}(\mathbb{B}^n)$, $f\neq 0$, the inequality (\ref{eq:carlmeasdefequiv}) holds.
 \item[(b)] For any $0<\delta<1$, the average function $\hat {\mu}_\delta$ belongs to $L^{\Phi_3}(\mathbb B^n, d\nu_\alpha)$.
 \item[(c)] The Berezin transform $\tilde {\mu}$  of the measure $\mu$ belongs to $L^{\Phi_3}(\mathbb B^n, d\nu_\alpha)$.

 \end{itemize}
 \end{theorem}
\vskip .1cm
We observe that the condition (ii) insures that the growth function $\Phi_1\circ\Phi_2^{-1}$ belongs to $\mathscr{U}$. If $\Phi_1(t)=t^p$ and $\Phi_2(t)=t^q$, then $\frac{\Phi_1\circ\Phi_2^{-1}(t)}{t}=t^{\frac pq-1}$ and for this to be non-decreasing, one should have $q\le p$. Hence the above result is an extension of the embedding $I_\mu: A_\alpha^p(\mathbb{B}^n)\rightarrow L^q(\mathbb{B}^n,d\mu)$ when $0<q<p<\infty$.
\vskip .2cm
As mentioned previously, the power function version of the above result is due to D. Luecking \cite{luecking}. In his proof, Luecking used a method that appeals to the atomic decomposition of Bergman spaces and Khintchine's inequalities. We will use the same approach here. We will prove and use a generalization of Khintchine's inequalities to growth functions, and take advantage of the recent characterization of the atomic decomposition of Bergman-Orlicz spaces obtained in \cite{BBT}. There are some other technicalities that require a good understanding of the properties of growth functions.
\vskip .2cm 
For $\phi:\mathbb{B}^n\rightarrow\mathbb{B}^n$ holomorphic, the composition operator $C_\phi$ is the operator defined for any $f\in H(\mathbb{B}^n)$ by $$C_\phi(f)(z)=(f\circ\phi)(z),\quad z\in \mathbb{B}^n.$$
We say the operator $C_\phi: A_\alpha^{\Phi_1}(\mathbb{B}^n)\rightarrow L^{\Phi_2}(\mathbb{B}^n,d\mu)$ is bounded, if there is a constant $K>0$ such that for any $f\in A_\alpha^{\Phi_1}(\mathbb{B}^n)$, $f\neq 0$, \Be\label{eq:compbound}\int_{\mathbb{B}^n}\Phi_2\left(\frac{|C_\phi(f)(z)|}{\|f\|_{\Phi_1,\alpha}}\right)d\mu(z)\le K.\Ee
The characterization of the symbols $\phi$ such that (\ref{eq:compbound}) holds, relies essentially  on the characterization of Carleson embeddings for Bergman-Orlicz spaces (see \cite{Charpentier,Charpentiersehba}). 
\vskip .1cm
Let $\phi$ be as above, and let $\beta>-1$. Define the measure $\mu_{\phi,\beta}$ by $$\mu_{\phi,\beta}(E)=\nu_\beta(\phi^{-1}(E))$$ for any Borel set $E\subseteq \mathbb{B}^n$.
\vskip .1cm
Then it follows from classical arguments (see for example \cite{Charpentiersehba}) and Theorem \ref{theo:main} that the following holds.
\begin{corollary}
Let $\Phi_1,\Phi_2\in \mathscr{L}\cup\mathscr{U}$, $\alpha,\beta>-1$.  Assume that 
\begin{itemize}
\item[(i)] $\Phi_1\circ\Phi_2^{-1}$ satisfies the $\nabla_2$-condition;
\item[(ii)] $\frac{\Phi_1\circ\Phi_2^{-1}(t)}{t}$ is non-decreasing.
\end{itemize}
Let  $\Phi_3$ be the complementary function of $\Phi_1\circ\Phi_2^{-1}$.
Then for any holomorphic self mapping $\phi$ of $\mathbb{B}^n$, the following conditions are equivalent.
\begin{itemize}
\item[(a)] $C_\phi$ is bounded from $A_\alpha^{\Phi_1}(\mathbb{B}^n)$ to $A_\beta^{\Phi_2}(\mathbb{B}^n)$.
\item[(b)] The measure $\mu_{\phi,\beta}$ is a $\Phi_2$-Carleson measure for $A_\alpha^{\Phi_1}(\mathbb{B}^n)$.
\item[(c)] For any $0<\delta<1$, the function $z\rightarrow \frac{\mu_{\phi,\beta}(D(z,\delta))}{\nu_\alpha(D(z,\delta))}$ belongs to $L^{\Phi_3}(\mathbb B^n, d\nu_\alpha)$.
\end{itemize}
\end{corollary}
\vskip .2cm
The paper is organized as follows. In the next section, we present some useful tools needed to prove our result. Section 3 is devoted to the proof of our result . 

\vskip .2cm
As usual, given two positive quantities $A$ and $B$, the notation $A\lesssim B$ (resp. $A\gtrsim B$) means that there is an absolute
positive constant $C$ such that $A\le CB$ (resp. $A\ge CB$). When $A\lesssim B$ and $B\lesssim A$, we write $A\approx B$ and say $A$
and $B$ are equivalent. Finally, all over the text,  $C$ or $K$ will
denote a positive constants not necessarily the same at distinct
occurrences.

\section{Preliminary results}

In this section, we give some fundamental facts about growth functions, Bergman metric and Bergman-Orlicz spaces. These results are needed in the proof our result.
\subsection{Growth functions and H\"older-type inequality}
Let $\Phi$ be a $\mathcal C^1$ growth function. Then the lower and the upper indices of $\Phi$ are respectively defined by
$$a_\Phi:=\inf_{t>0}\frac{t\Phi^\prime(t)}{\Phi(t)}\,\,\,\textrm{and}\,\,\,b_\Phi:=\sup_{t>0}\frac{t\Phi^\prime(t)}{\Phi(t)}.$$
It is well known that if $\Phi$ is convex, then $1\le a_\Phi\le b_\Phi<\infty$ and, if $\Phi$ is concave, then $0<a_\Phi\le b_\Phi\le 1$. We observe that a convex growth function satisfies the $\bigtriangledown_2-$condition if and only if $1< a_\Phi\le b_\Phi<\infty$ (see \cite[Lemma 2.6]{DHZZ}). 
\vskip .2cm
We also observe that if $\Phi$ is a $\mathcal C^1$ growth function,  then the functions $\frac{\Phi(t)}{t^{a_\Phi}}$ and $\frac{\Phi^{-1}(t)}{t^{\frac{1}{b_\Phi}}}$ are increasing. We then deduce the following.

\begin{lemma}\label{lem:phip}
Let $\Phi\in \mathscr{L}_p$. Then the growth function $\Phi_p$, defined by
$\Phi_p(t)=\Phi(t^{1/p})$ is in $\mathscr{U}^q$ for some $q\geq 1$. 
\end{lemma}
\vskip .1cm
When writing $\Phi\in \mathscr{U}^q$ (resp. $\Phi\in \mathscr{L}_p$), we will always assume that $q$ (resp. $p$) is the smallest (resp. biggest) number $q_1$ (resp. $p_1$) such that $\Phi$ is of upper type $q_1$ (resp. lower type $p_1$). We note that $a_\Phi$ (resp. $b_\Phi$) coincides with the biggest (resp. smallest) number $p$ such that $\Phi$ is of lower (resp. upper) type $p$.
\vskip .1cm
We also make the following observation (see \cite[Proposition 2.1]{sehbatchoundja}).
\begin{proposition}\label{phiandinverse}
The following assertion holds:
\begin{center}
    $\Phi\in \mathscr{L}$ if and only if $\Phi^{-1}\in \mathscr{U}.$
\end{center}
\end{proposition}
We recall the following H\"older-type inequality (see \cite[Page 58]{raoren}).
\begin{lemma}\label{lem:holdergenecompl}
Let $\Phi\in \mathscr{U}$, $\alpha>-1$. Denote by $\Psi$ the complementary function of $\Phi$. Then
$$\int_{\mathbb{B}^n}|f(z)g(z)|d\nu_\alpha(z)\le 2\left(\int_{\mathbb{B}^n}\Phi\left(|f(z)|\right)d\nu_\alpha(z)\right)\left(\int_{\mathbb{B}^n}\Psi\left(|g(z)|\right)d\nu_\alpha(z)\right).$$
\end{lemma} 

\subsection{The Bergman metric and atomic decomposition}
For $a\in \mathbb B^n$, $a\ne 0$, let $\varphi_a$ denote
the automorphism of $\mathbb B^n$ taking $0$ to $a$ defined
by
$$\varphi_{a}(z)=\frac{a - P_{a}(z) - (1 - |z|^{2})^{\frac{1}{2}}Q_{a}(z)}{1 -
 \langle z,a\rangle}$$ where $P_a$ is the projection of $\mathbb C^n$ onto the
 one-dimensional subspace span of $a$ and $Q_a=I - P_a$ where
 $I$ is the identity.
 It is easy to see that
 $$\varphi_a(0)=a,\,\,\,
 \varphi_a(a)=0,\,\,\, \varphi_a o \varphi_a(z)=z.$$
 The Bergman metric $d$ on the unit ball $\mathbb{B}^n$ is defined by $$d(z,w)=\frac 12\log\left(\frac{1+\varphi_z(w)}{1-\varphi_z(w)}\right),\,z,w\in \mathbb{B}^n.$$

For $\delta>0$, we denote by $$D(z,\delta)=\{w\in \mathbb {B}^n: d(z,w)<\delta\},$$ the Bergman ball centered at $z$ with radius $\delta.$ It is well known that for any $\alpha>-1$, and for $w\in D(z,\delta)$, $$\nu_\alpha(D(z,\delta))\approx |1-\langle z,w\rangle|^{n+1+\alpha}\approx(1-|w|^2)^{n+1+\alpha}.$$
We recall that a sequence $a=\{a_k\}_{k\in \mathbb{N}}$ in $\mathbb{B}^n$ is said to be separated in the Bergman metric $d$, if there is a positive constant $r>0$ such that $$d(a_k,a_j)>r\,\,\,\textrm{for}\,\,\,k\neq j.$$
We refer to \cite[Theorem 2.23]{Zhu} for the following result.

\begin{theorem}\label{thm:covering}
Given $\delta\in (0, 1)$, there exists a sequence $\{a_k\}$ of points  of $\mathbb{B}^n$ called $\delta$-lattice such that

\begin{itemize}
\item[(i)] the balls $D(a_k,\frac{\delta}{4})$ are pairwise disjoint;
\item[(ii)]$\mathbb{B}^n=\cup_{k}D(a_k,\delta)$;
\item[(iii)] there is an integer N (depending only on $\mathbb{B}^n$) such
that each point of $\mathbb{B}^n$ belongs to at most N of the balls $D(a_k,4\delta)$.
\end{itemize}
\end{theorem}

\vskip .2cm
Let $\alpha>-1$ and $\Phi$ a growth function. For a fixed $\delta$-lattice $a=\{a_k\}_{k\in \mathbb{N}}$ in $\mathbb{B}^n$, we define by $l_{a,\alpha}^{\Phi}$, the space of all complex sequences $c=\{c_k\}_{k\in \mathbb{N}}$ such that
$$\sum_k(1-|a_k|^2)^{n+1+\alpha}\Phi(|c_k|)<\infty.$$
The following was observed in \cite[Proposition 2.8]{TZ}.
\begin{lemma}\label{lem:dualtityseq} Let $\Phi\in \mathscr{U}$, and $\alpha>-1$. Assume that $\Phi$ satisfies the $\nabla_2$-condition and denote by $\Psi$ its complementary function. Then, the dual space $\left(\ell_{a,\alpha}^{\Phi}\right)^*$ of the space
$\ell_{a,\alpha}^{\Phi}$ identifies with $\ell_{a,\alpha}^{\Psi}$
under the sum pairing $$\langle c,d\rangle_{a,\alpha}=
\sum_{k}(1-|a_k|^2)^{n+1+\alpha}c_k\overline {d}_k,$$ where $c= \{c_{k}\}_{k\in \mathbb{N}}\in l_{a,\alpha}^{\Phi}$, $d =\{d_{k}\}_{k\in \mathbb{N}}\in l_{a,\alpha}^{\Psi}$ for $a=\{a_k\}_{k\in \mathbb{N}}$ a fixed $\delta$-lattice in $\mathbb{B}^n$.
\end{lemma}
For $\Phi$ a growth function, define \Be\label{eq:pPhi} p_\Phi= \left\{\begin{array}{lcr}1 & \mbox{if} & \Phi\in \mathscr{U}\\ p & \mbox{if} & \Phi\in \mathscr{L}_p.\end{array}\right.\Ee
We refer to \cite[Proposition 3.1 and Proposition 3.2]{BBT} for the following result.
\begin{proposition}\label{prop:atomicdecomp}
Let $\Phi\in \mathscr{L}\cup\mathscr{U}$, $\alpha>-1$ and $b>\frac{n+1+\alpha}{p_\Phi}$. Let $a=\{a_k\}_{k\in \mathbb{N}}$ be a $\delta$-separated sequence in $\mathbb{B}^n$. Then for any sequence $c=\{c_k\}_{k\in\mathbb{N}}$ of complex numbers that satisfies the condition $$\sum_k(1-|a_k|^2)^{n+1+\alpha}\Phi\left(\frac{|c_k|}{(1-|a_k|^2)^b}\right)<\infty,$$
the series $\sum_{k=1}^\infty\frac{c_k}{(1-\langle z,a_k\rangle)^b}$ converges in
 $A_\alpha^{\Phi}(\mathbb{B}^n)$  to a function $f$ and 
\Be\label{eq:atomicdecomp}
\int_{\mathbb{B}^n}\Phi\left(\left|\sum_{k=1}^\infty\frac{c_k}{(1-\langle z,a_k\rangle)^b}\right|\right)d\nu_\alpha(z)\lesssim \sum_k(1-|a_k|^2)^{n+1+\alpha}\Phi\left(\frac{|c_k|}{(1-|a_k|^2)^b}\right).
\Ee
\end{proposition}

Let us close this subsection with the following estimate.
\begin{lemma}\label{lem:meanvalue}
Let $r>0$, $\Phi\in \mathscr{L}\cup\mathscr{U}$ and $\alpha>-1$. Then there exists a constant $K>0$ such that for any $f\in H(\mathbb{B}^n)$,
\begin{equation}\label{eq:meanvalue}
\Phi(|f(z)|)\le K\int_{D(z,r)}\Phi(|f(\zeta)|)\frac{d\nu(w)}{(1-|w|^2)^{n+1}}.
\end{equation}
\end{lemma}
\begin{proof}
For $\Phi$ a growth function, define $p_\Phi$ as in (\ref{eq:pPhi}). Then by \cite[Lemma 2.24]{Zhu}, there exists a constant $C>0$ such that
\Beas
|f(z)|^{p_\Phi}&\le& \frac{C}{(1-|z|)^{n+1+\alpha}}\int_{D(z,r)}|f(\zeta)|^{p_\Phi}d\nu_\alpha(w)\\ &\lesssim& \int_{D(z,r)}|f(\zeta)|^{p_\Phi}\frac{d\nu_\alpha(w)}{\nu_\alpha(D(z,r))}.
\Eeas
Observing that $\Phi_p(t)=\Phi(t^{\frac{1}{p_\Phi}})$ is in $\mathscr{U}$, we obtain applying Jensen's inequality to the above estimate that
\Beas
\Phi(|f(z)|)&\le& C \int_{D(z,r)}\Phi(|f(\zeta)|)\frac{d\nu_\alpha(w)}{\nu_\alpha(D(z,r))}\\ &\le& K\int_{D(z,r)}\Phi(|f(\zeta)|)\frac{d\nu(w)}{(1-|w|^2)^{n+1}}.
\Eeas

\end{proof}

\subsection{Averaging functions and Berezin transform}
Let $\mu$ be a positive measure on $\mathbb{B}^n$ and $\alpha>-1$. 
For $w\in \mathbb{B}^n$, the normalized reproducing kernel at $w$ is given by
\begin{equation*}\label{eq:normalrepkern} k_{\alpha,w}(z)=\frac{(1-|w|^2)^{\frac{n+1+\alpha}{2}}}{(1-\langle z,w\rangle)^{n+1+\alpha}}.
\end{equation*}
The Berezin transform $\tilde {\mu}$  of the measure $\mu$ is the function defined for any $w\in\mathbb{B}^n$ by $$\tilde {\mu}(w):=\int_{\mathbb{B}^n}|k_{\alpha,w}(z,w)|^2d\mu(z).$$
When $d\mu(z)=f(z)d\nu_\alpha(z)$, we write $\tilde {\mu}=\tilde {f}$ and speak of the Berezin transform of the function $f$.
\vskip .2cm
It is well known that the Berezin transform ($f\mapsto \tilde{f}$) is bounded on $L_\alpha^p(\mathbb{B}^n)$ if and only if $p>1$ (see \cite{Zhao1}). It follows from this and the interpolation result \cite[Theorem 4.3]{DHZZ} that the following holds.
\begin{lemma}\label{lem:berezinbound}
Let $\alpha>-1$. Then the Berezin transform is bounded on $L_\alpha^{\Phi}(\mathbb{B}^n)$ for any $\Phi\in \mathscr{U}$ that satisfies the $\nabla_2$-condition.
\end{lemma}
For $z\in \mathbb{B}^n$ and $\delta\in (0,1)$, we define the average of the positive measure $\mu$ at $z$ by $$\hat {\mu}_\delta(z)=\frac{\mu(D(z,\delta))}{\nu_\alpha(D(z,\delta))}.$$
The function $\hat {\mu}$ is very useful in the characterization of Carleson embeddings with loss and some other operators (see \cite{luecking1,Zhu2,Zhu3} and the references therein).

The following lemma follows as in the power functions case \cite[Proposition 3.6]{Choe} (see also \cite[Lemma 2.9]{nanasehba}).
\begin{lemma}\label{lem:variationoflattice}
Let $\Phi\in \mathscr{L}\cup\mathscr{U}$, $\alpha>-1$, and $r,s\in (0,1)$. Assume that $\mu$ is a positive Borel measure on $\mathbb{B}^n$. Then
the following assertions are equivalent.
\begin{itemize}
\item[(i)] The function  $\mathbb{B}^n\ni z\mapsto \frac{\mu(D(z,r))}{\nu_\alpha(D(z,r))}$ belongs to $L_\alpha^{\Phi}(\mathbb{B}^n)$.
\item[(ii)] The function  $\mathbb{B}^n\ni z\mapsto \frac{\mu(D(z,s))}{\nu_\alpha(D(z,s))}$ belongs to $L_\alpha^{\Phi}(\mathbb{B}^n)$.\\
\end{itemize}
\end{lemma}
The following is an elementary exercise (see \cite[Page 16]{Zhu3}).
\begin{lemma}\label{lem:Berezinsupaverage}
Let $\delta\in (0,1)$. Then there exists a constant $C_\delta>0$ such that
$$\hat {\mu}_\delta(z)\le C_\delta\tilde {\mu}(z),\,\,\,\textrm{for any}\,\,\,z\in \mathbb{B}^n.$$
\end{lemma}
We also observe the following.
\begin{lemma}\label{lem:AveragesupBerezin}
Given $\delta\in (0,1)$, there exists $C_\delta>0$ such that
$$\tilde {\mu}(z)\le C_\delta\Tilde {\hat {\mu}}_\delta(z),\,\,\,\textrm{for any}\,\,\,z\in \mathbb{B}^n.$$
\end{lemma}
\begin{proof}
Using Lemma \ref{lem:meanvalue}
and Fubini's lemma, we obtain
\Beas
\tilde {\mu}(z) &=& \int_{\mathbb{B}^n}|k_{\alpha,z}(\xi)|^2d\mu(\xi)\\ &\lesssim& \int_{\mathbb{B}^n} \frac{1}{\nu_\alpha(D(\xi,\delta))}\int_{D(\xi,\delta)}|k_{\alpha,z}(w)|^2d\nu_\alpha(w)d\mu(\xi)\\ &=& \int_{\mathbb{B}^n} \frac{1}{\nu_\alpha(D(\xi,\delta))}\int_{\mathbb{B}^n}|k_{\alpha,z}(w)|^2\chi_{D(\xi,\delta)}(w)d\nu_\alpha(w)d\mu(\xi)\\ &=& \int_{\mathbb{B}^n} \frac{1}{\nu_\alpha(D(\xi,\delta))}\int_{\mathbb{B}^n}|k_{\alpha,z}(w)|^2\chi_{D(w,\delta)}(\xi)d\nu_\alpha(w)d\mu(\xi)\\ &\approx& \int_{\mathbb{B}^n}\frac{\mu(D(w,\delta))}{\nu_\alpha(D(\xi,\delta)}|k_{\alpha,z}(w)|^2d\nu_\alpha(w)\\ &=& \Tilde {\hat {\mu}}_\delta(z).
\Eeas
We have used that $\chi_{D(\xi,\delta)}(w)=\chi_{D(w,\delta)}(\xi)$ and that as $w\in D(\xi,\delta)$, $\nu_\alpha(D(\xi,\delta))\approx \nu_\alpha(D(w,\delta))$.
\end{proof}
We then have the following result.

\begin{lemma}\label{lem:integraldiscretizationAverBer}
Let $\Phi\in \mathscr{L}\cup\mathscr{U}$, $\alpha>-1$, and $r,s\in (0,1)$. Let $a=\{a_j\}_{j\in \mathbb N}$ be a $r$-lattice in $\mathbb{B}^n$. Then the following assertions are equivalent.
\begin{itemize}
\item[(i)] $\hat {\mu}_s\in L_\alpha^\Phi(\mathbb{B}^n)$.
\item[(ii)] $\{\hat {\mu}_r(z_j)\}_{j\in \mathbb N}\in l_{a,\alpha}^\Phi$.\\
If moreover, $\Phi\in \mathscr{U}$ and satisfies the $\nabla_2$-condition, then the above assertions are both equivalent to the following,
\item[(iii)] The Berezin transform $\tilde {\mu}$  of the measure $\mu$ belongs to $L^{\Phi}(\mathbb B^n, d\nu_\alpha)$
\end{itemize}
\end{lemma}
\begin{proof}
The equivalence (i)$\Leftrightarrow$(ii) follows as in the power functions case (see \cite[Theorem 3.9]{Choe} or \cite[Lemma 2.12]{nanasehba} ).
That (i)$\Rightarrow$(ii) is Lemma \ref{lem:AveragesupBerezin} together with Lemma \ref{lem:berezinbound}. That (iii)$\Rightarrow$(i) follows from Lemma \ref{lem:Berezinsupaverage}. 
\end{proof}

\section{Carleson measures for weighted Bergman-Orlicz spaces}
In this section, we present the proof of our characterization of Carleson measures for weighted Bergman-Orlicz spaces. We start by recalling the classical Khinchine's inequality. We recall that the Rademacher functions $r_n$ in $(0,1]$ are defined as follows
$$
r_n(t)=sgn(\sin 2^n\pi t),\,n=0,1,2\ldots
$$

\begin{lemma}[Khinchine's inequality]\label{lem:Kinchine}
For $0 < p <\infty$ there exist constants $0 <A_p\le B_p <\infty$ such that for any sequence of complex numbers $x=\{x_k\}\in \ell^2$,

\Be\label{eq:kinchine}A_p\left(\sum_{k}|x_k|^2\right)^{1/2}\le \left(\int_0^1\left|\sum_{k}x_kr_k(t)\right|^pdt\right)^{1/p}\le B_p\left(\sum_{k}|x_k|^2\right)^{1/2}.\Ee
\end{lemma}
We next show how to extend this result to the case where the power function $t^p$ is replaced by any growth function in the classes considered in this note.
\begin{theorem}[Extended Khinchine's inequality]\label{lem:Kinchine}
For $\Phi\in \mathscr{L}\cup\mathscr{U}$, there exist constants $0 <A_\Phi\le B_\Phi <\infty$ such that, for any sequence $x=\{x_k\}\in \ell^2$,

\Be\label{eq:kinchinegene}A_\Phi\left(\sum_{k}|x_k|^2\right)^{1/2}\le \Phi^{-1}\left(\int_0^1\Phi\left(\left|\sum_{k}x_kr_k(t)\right|\right)dt\right)\le B_\Phi\left(\sum_{k}|x_k|^2\right)^{1/2}.\Ee
\end{theorem}
\begin{proof}
Let us first prove the left inequality in (\ref{eq:kinchinegene}). If $\Phi\in \mathscr{L}\cup\mathscr{U}$, then if $p_\Phi$ is defined as in (\ref{eq:pPhi}), we recall with Lemma \ref{lem:phip} that the growth function $\Phi_{p_\Phi}(t)=\Phi(t^{\frac 1{p_\Phi}})$ belongs to $\mathscr{U}$. Then applying $\Phi$ to the left inequality in (\ref{eq:kinchine}) with $p=p_\Phi$, and using Jensen's inequality, we obtain
\Beas
\Phi\left(A_p\left(\sum_{k}|x_k|^2\right)^{1/2}\right)&\le& \Phi\left(\left(\int_0^1\left|\sum_{k}x_kr_k(t)\right|^pdt\right)^{1/p}\right)\\ &=& \Phi_{p_\Phi}\left(\int_0^1\left|\sum_{k}x_kr_k(t)\right|^pdt\right)\\ &\le& \int_0^1\Phi_{p_\Phi}\left(\left|\sum_{k}x_kr_k(t)\right|^p\right)dt\\ &=& \int_0^1\Phi\left(\left|\sum_{k}x_kr_k(t)\right|\right)dt.
\Eeas
Let us now prove the right inequality. If $\Phi\in \mathscr{L}$, then by Proposition \ref{phiandinverse}, $\Phi^{-1}\in \mathscr{U}$. It follows using the Jensen's inequality and the right hand side of (\ref{eq:kinchine}) with $p=1$ that
\Beas
\Phi^{-1}\left(\int_0^1\Phi\left(\left|\sum_{k}x_kr_k(t)\right|\right)dt\right) &\le& \int_0^1\left|\sum_{k}x_kr_k(t)\right|dt\\ &\le& B_1\left(\sum_{k}|x_k|^2\right)^{1/2}.
\Eeas
If $\Phi\in \mathscr{U}$, and if $q$ is its upper indice, then the growth function $\Phi_q(t)=\Phi(t^{1/q})$ belongs to $\mathscr{L}$. Hence using the right hand side of (\ref{eq:kinchine}) with $p=q$, we first obtain
\Beas
\Phi_q\left(\int_0^1\left|\sum_{k}x_kr_k(t)\right|^qdt\right) &=& \Phi\left(\left(\int_0^1\left|\sum_{k}x_kr_k(t)\right|^qdt\right)^{1/q}\right)\\ &\le& \Phi\left(B_q\left(\sum_{k}|x_k|^2\right)^{1/2}\right).
\Eeas
As $\Phi_q$ is concave, it follows that 
\Beas
\int_0^1\Phi_q\left(\left|\sum_{k}x_kr_k(t)\right|^q\right)dt &\le& \Phi_q\left(\int_0^1\left|\sum_{k}x_kr_k(t)\right|^qdt\right).
\Eeas
Hence 
$$\int_0^1\Phi\left(\left|\sum_{k}x_kr_k(t)\right|\right)dt\le \Phi\left(B_q\left(\sum_{k}|x_k|^2\right)^{1/2}\right).$$
The proof is complete.
\end{proof}
\begin{proof}[Proof of Theorem \ref{theo:main}]
We observe that the equivalence (b)$\Leftrightarrow$(c) is given in Lemma \ref{lem:integraldiscretizationAverBer}. It is then enough to prove that (a)$\Leftrightarrow$(b). This is given in the result below.
\end{proof}
We have the following embedding with loss.
\begin{theorem}\label{theo:main} Let $\Phi_1,\Phi_2\in \mathscr{L}\cup\mathscr{U}$, $\alpha>-1$.  Assume that 
\begin{itemize}
\item[(i)] $\Phi_1\circ\Phi_2^{-1}$ satisfies the $\nabla_2$-condition;
\item[(ii)] $\frac{\Phi_1\circ\Phi_2^{-1}(t)}{t}$ is non-decreasing.
\end{itemize}
Let $\mu$ is a positive measure on $\mathbb B^n$, and let  $\Phi_3$ be complementary function of $\Phi_1\circ\Phi_2^{-1}$. Then the following assertions are satisfied.
 \begin{itemize}
 \item[(a)] If the function
$$\mathbb{B}^n\ni z\mapsto \frac{\mu(D(z,\delta))}{\nu_\alpha(D(z,\delta))}$$
belongs to $L_\alpha^{\Phi_3}(\mathbb{B}^n)$,  for some $\delta\in (0,1)$, then
there exists a constant $C>0$ such that for any $f\in A_\alpha^{\Phi_1}(\mathbb{B}^n)$, with $f\neq 0$,
\begin{equation}\label{eq:Carlembed21}
\int_{\mathbb{B}^n}\Phi_2\left(\frac{|f(z)|}{\|f\|_{\Phi_1,\alpha}^{lux}}\right)d\mu (z)\le C.
\end{equation}
 \item[(b)] If (\ref{eq:Carlembed21}) holds, then for any $0<\delta<1$, the average function $\hat {\mu}_\delta$ belongs to $L^{\Phi_3}(\mathbb B^n, d\nu_\alpha)$.
 \end{itemize}
 \end{theorem}

\begin{proof}
Let us start with the proof of assertion (a). Let $K$ be the constant in (\ref{eq:meanvalue}). Then for $1>r>\delta>0$ fixed, we obtain using  (\ref{eq:meanvalue}) that
\Beas M &:=& \int_{\mathbb{B}^n}\Phi_2\left(\frac{|f(z)|}{\|f\|_{\Phi_1,\alpha}^{lux}}\right)d\mu (z)\\ &\le& K\int_{\mathbb{B}^n}\left(\int_{D(z,\delta)}\Phi_2\left(\frac{|f(z)|}{\|f\|_{\Phi_1,\alpha}^{lux}}\right)\frac{d\nu(w)}{(1-|w|^2)^{n+1}}\right)d\mu(z)\\ &=& K\int_{\mathbb{B}^n}\left(\int_{\mathbb{B}^n}\chi_{D(z,\delta)}(w)d\mu(z)\right)\Phi_2\left(\frac{|f(z)|}{\|f\|_{\Phi_1,\alpha}^{lux}}\right)\frac{d\nu(w)}{(1-|w|^2)^{n+1}}\\ &\le& K\int_{\mathbb{B}^n}\Phi_2\left(\frac{|f(w)|}{\|f\|_{\Phi_1,\alpha}^{lux}}\right)\frac{\mu(D(w,r))}{(1-|w|^2)^{n+1+\alpha}}d\nu_\alpha(w).
\Eeas
It follows from Lemma \ref{lem:holdergenecompl} and the hypothesis that
\Beas
M 
&\le& 2K\left(\int_{\mathbb{B}^n}\Phi_1\left(\frac{|f(w)|}{\|f\|_{\Phi_1,\alpha}^{lux}}\right)d\nu_\alpha(w)\right)\left(\int_{\mathbb{B}^n}\Phi_3\left(\frac{\mu(D(w,r))}{(1-|w|^2)^{n+1+\alpha}}\right)d\nu_\alpha(w)\right)\\ &\le& 2K\left(\int_{\mathbb{B}^n}\Phi_3\left(\frac{\mu(D(w,r))}{(1-|w|^2)^{n+1+\alpha}}\right)d\nu_\alpha(w)\right).
\Eeas
Hence (\ref{eq:Carlembed21}) holds with constant $C=2K\left(\int_{\mathbb{B}^n}\Phi_3\left(\frac{\mu(D(w,r))}{(1-|w|^2)^{n+1+\alpha}}\right)d\nu_\alpha(w)\right).$
\vskip .2cm
Proof of (b): Assume that (\ref{eq:Carlembed21}) holds for any $0\neq f\in A_\alpha^{\Phi_1}(\mathbb{B}^n)$. 
We recall with Proposition \ref{prop:atomicdecomp} that for any sequence $c=\{c_k\}_{k\in\mathbb{N}}$ of complex numbers that satisfies the condition \Be\label{eq:phiseq}\sum_k(1-|a_k|^2)^{n+1+\alpha}\Phi\left(\frac{|c_k|}{(1-|a_k|^2)^b}\right)<\infty,\Ee where $a=\{a_k\}$ is some $\delta$-lattice in $\mathbb{B}^n$, and $b>\frac{n+1+\alpha}{p_\Phi}$,
the series $\sum_{k=1}^\infty\frac{c_k}{(1-\langle z,a_k\rangle)^b}$ converges in
 $A_\alpha^{\Phi}(\mathbb{B}^n)$  to a function $f$. 
\vskip .1cm
For simplicity, we may assume that $f$ is such that $\|f\|_{\Phi_1,\alpha}^{lux}=1$. Thus 
$$\int_{\mathbb{B}^n}\Phi_2\left(\left|\sum_{k=1}^\infty\frac{c_k}{(1-\langle z,a_k\rangle)^b}\right|\right)d\mu(z)\le C.$$
Replacing $c_k$ by $c_kr_k(t)$, this gives us
\Be\label{eq:stepkinchine}\int_{\mathbb{B}^n}\Phi_2\left(\left|\sum_{k=1}^\infty\frac{c_kr_k(t)}{(1-\langle z,a_k\rangle)^b}\right|\right)d\mu(z)\le C.\Ee
By the extended Kinchine's inequalities, we have 
$$\Phi_2\left(A_{\Phi_2}\left(\sum_{k=1}^\infty\frac{|c_k|^2}{|1-\langle z,a_k\rangle|^{2b}}\right)^{1/2}\right)\le \int_0^1\Phi_2\left(\left|\sum_{k=1}^\infty\frac{c_kr_k(t)}{(1-\langle z,a_k\rangle)^b}\right|\right)dt.$$
We can also assume that $A_{\Phi_2}=1$. From the last inequality and (\ref{eq:stepkinchine}), we obtain
\Be\label{eq:normalizedineq}\int_{\mathbb{B}^n}\Phi_2\left(\left(\sum_{k=1}^\infty\frac{|c_k|^2}{|1-\langle z,a_k\rangle|^{2b}}\right)^{1/2}\right)d\mu(z)\le C.\Ee

Put $d_k=\frac{c_k}{(1-\langle z,a_k\rangle)^{b}}$ and observe that the sequence $\{\Phi_2(|d_k|)\}$ belongs to $\ell_{a,\alpha}^{\Phi_1\circ\Phi_2^{-1}}$ whenever $c=\{c_k\}$ satisfies (\ref{eq:phiseq}). Define $D_k=D(a_k,\delta)$. Then
\Bea\label{eq:innerseq} \nonumber\left\langle\{\Phi_2(|d_k|)\},\left\{\frac{\mu(D_k)}{(1-|a_k|^2)^{n+1+\alpha}}\right\}\right\rangle_\alpha &=& \sum_k\Phi_2(|d_k|)\mu(D_k)\\ &=& \int_{\mathbb{B}^n}\sum_k\Phi_2(|d_k|)\chi_{D_k}(z)d\mu(z).\Eea

We observe that if $\tilde{\Phi}_2(t)=\Phi_2(t^{1/2})$ is in $\mathscr{U}$, then
\Beas
\sum_k\Phi_2(|d_k|)\chi_{D_k}&\le& \tilde{\Phi}_2\left(\sum_k|d_k|^2\chi_{D_k}\right)\\ &\le& \Phi_2\left(\left(\sum_k|d_k|^2\chi_{D_k}\right)^{1/2}\right).
\Eeas
If $\tilde{\Phi}_2\in \mathscr{L}_s$, then as $\tilde{\Phi}_2(t^{1/s})$ is in $\mathscr{U}$, we obtain
\Beas
\sum_k\Phi_2(|d_k|)\chi_{D_k}&\le& \tilde{\Phi}_2\left(\left(\sum_k|d_k|^{2s}\chi_{D_k}\right)^{\frac 1s}\right)\\ &\le& \tilde{\Phi}_2\left[\left(\sum_k|d_k|^{2}\chi_{D_k}\right)\left(\sum_k\chi_{D_k}\right)^{\frac{1-s}{s}}\right]\\ &\le& \tilde{\Phi}_2\left(N^{\frac{1-s}{s}}\left(\sum_k|d_k|^{2}\chi_{D_k}\right)\right)\\ &\le& K{\Phi}_2\left(\left(\sum_k|d_k|^{2}\chi_{D_k}\right)^{\frac{1}{2}}\right).
\Eeas

Taking these observations in (\ref{eq:innerseq}) and using (\ref{eq:normalizedineq}), we obtain
\Beas
 {L} &:=& \langle \{\Phi_2(|d_k|)\},\left\{\frac{\mu(D_k)}{(1-|a_k|^2)^{n+1+\alpha}}\right\}\rangle_\alpha\\ &\le& K\int_{\mathbb{B}^n}{\Phi}_2\left(\left(\sum_k|d_k|^{2}\chi_{D_k}(z)\right)^{\frac{1}{2}}\right)d\mu(z)\\ &=& K\int_{\mathbb{B}^n}{\Phi}_2\left(\left(\sum_k\frac{|c_k|^{2}}{(1-|a_k|^2)^{2b}}\chi_{D_k}(z)\right)^{\frac{1}{2}}\right)d\mu(z)\\ &\le& K\int_{\mathbb{B}^n}{\Phi}_2\left(\left(\sum_k\frac{|c_k|^{2}}{(1-|a_k|^2)^{2b}}\frac{(1-|a_k|^2)^{2b}}{|1-\langle z,a_k\rangle^2|^{2b}}\right)^{\frac{1}{2}}\right)d\mu(z)\\ &=& K\int_{\mathbb{B}^n}{\Phi}_2\left(\left(\sum_k\frac{|c_k|^{2}}{|1-\langle z,a_k\rangle^2|^{2b}}\right)^{\frac{1}{2}}\right)d\mu(z)\\ &\le& KC.
\Eeas
As this holds for any the sequence $\{\Phi_2(|d_k|)\}$ belonging to $\ell_{a,\alpha}^{\Phi_1\circ\Phi_2^{-1}}$, we deduce that the sequence $\{\frac{\mu(D_k)}{(1-|a_k|^2)^{n+1+\alpha}}\}$ belongs to $\ell_{a,\alpha}^{\Phi_3}$. By Lemma \ref{lem:integraldiscretizationAverBer}, this is equivalent to saying that  the average function $\hat {\mu}_\delta$ belongs to $L^{\Phi_3}(\mathbb B^n, d\nu_\alpha)$.
The proof is complete.
\end{proof}

\end{document}